\documentclass[11pt,reqno,a4paper]{amsart}
\usepackage{euscript}

\newtheorem{theorem}{Theorem}
\newtheorem{proposition}{Proposition}
\newtheorem{corollary}{Corollary}

\newtheorem{remark}{Remark}
\newtheorem{example}{Example}
\renewcommand{\epsilon}{\varepsilon}
\renewcommand{\phi}{\varphi}

\DeclareMathOperator{\Ker}{Ker}
\def\Z{\mathbb{Z}}
\def\R{\mathbb{R}}
\def\cA{\EuScript{A}}

\def\Id{\text{\rm Id}}

\begin{document}

\title[Shadowing for infinite dimensional dynamics]{Shadowing for infinite dimensional dynamics and exponential trichotomies 
}

\begin{abstract}
Let $(A_m)_{m\in \Z}$ be a sequence of bounded linear maps acting on an arbitrary Banach space $X$ and admitting an exponential trichotomy and let $f_m:X\to X$ be a Lispchitz map for every $m\in \Z$. We prove that whenever the Lipschitz constants of $f_m$, $m\in \Z$, are uniformly  small, the nonautonomous dynamics given by $x_{m+1}=A_mx_m+f_m(x_m)$, $m\in \Z$, has various types of shadowing. Moreover, if $X$ is finite dimensional and each $A_m$ is invertible we prove that a converse result is also true. Furthermore, we get similar results for one-sided and continuous time dynamics. As applications of our results we study the  Hyers-Ulam stability for certain difference equations and we obtain a  very general version of the Grobman-Hartman's theorem for nonautonomous dynamics. 
\end{abstract}

\author{Lucas Backes}
\address{\noindent Departamento de Matem\'atica, Universidade Federal do Rio Grande do Sul, Av. Bento Gon\c{c}alves 9500, CEP 91509-900, Porto Alegre, RS, Brazil.}
\email{lhbackes@impa.br} 

\author{Davor Dragi\v cevi\'c}
\address{Department of Mathematics, University of Rijeka, Croatia}
\email{ddragicevic@math.uniri.hr}

\maketitle

\maketitle

\section{Introduction}

The foundations of the theory of chaotic dynamical systems dates back to the work of Poincar\'e \cite{Poi90} and is now a well developed area of research. An important feature of chaotic dynamical systems, already observed by Poincar\'e, is the sensitivity to initial conditions: any small change to the initial condition may lead to a large discrepancy in the output. This fact makes somehow complicated or even impossible the task of predicting the real trajectory of the system based on approximations. On the other hand, many chaotic systems, like uniformly hyperbolic dynamical systems \cite{An70, Bow75}, exhibit an amazing property stating that, even though a small error in the initial condition may led eventually to a large effect, there exists a true orbit with a slightly different initial condition that stays near the approximate trajectory. This property is known as the \emph{shadowing property}.

The objective of this paper is to develop a shadowing theory for \emph{nonautonomous} systems acting on an arbitrary Banach space $X$. More precisely, starting with a linear dynamics
\begin{equation}\label{eq: intr}
x_{m+1}=A_m x_m \quad m\in \Z, 
\end{equation}
where the sequence $(A_m)_{m\in \Z}$ admits an \emph{exponential trichotomy}, we prove that a small \emph{nonlinear} perturbation of~\eqref{eq: intr} has various types of the shadowing properties. Moreover, if $X$ has finite dimension and the linear maps $A_m$ are invertible, we prove that $(A_m)_{m\in \Z}$ admits an exponential trichotomy whenever \eqref{eq: intr} satisfies a certain type of shadowing. Furthermore, we partially extend these results to one-sided dynamics and to continuous time dynamics. As applications of our results we provide a characterization of Hyers-Ulam stability for certain difference equations and also exhibit a very general version of the Grobman-Hartman's theorem for nonautonomous dynamics.

\subsection{Relations with previous results}
Our proof is inspired by the analytical proofs of the shadowing lemma by Palmer~\cite{Pal88} and Mayer and Sell~\cite{MS87}. These proofs have also inspired versions of the shadowing lemma for maps acting on Banach spaces (see~\cite{CLP89, Hen94}). While all these previous results deal with \emph{autonomous} dynamical systems, we on the other hand focus in the nonautonomous setting. 

Our work was initiated in \cite{BD19}. In that paper we have also dealt with the shadowing problem in the nonautonomous realm but in much less generality. In fact, our Theorem \ref{cor: uniqueness} generalizes the main result of \cite{BD19} in three directions: \begin{itemize} 
\item  we allow the sequence $(A_m)_{m\in \Z}$ to admit an exponential trichotomy instead of more restrictive assumption made in~\cite{BD19} that $(A_m)_{m\in \Z}$ admits an exponential dichotomy; \item  the nonlinear perturbations of \eqref{eq: intr} allowed here are much more general (for instance, they do not need to be differentiable or bounded as in~\cite{BD19});
\item in the present paper, we don't assume that $\sup_{m\in \Z}\lVert A_m\rVert<\infty$.
\end{itemize}
Moreover, 
in the present paper we treat the cases of one-sided dynamics and continuous time dynamics that were not considered in the previous work allowing us, for instance, to characterize Hyers-Ulam stability for certain difference equations. 

In order to finish this introduction, we would also like to stress that there are other shadowing results for nonautonomous dynamics in Banach spaces in the literature (see for instance the nice monographs \cite{Pal00, Pil99}) but in all those results there are some differentiability and boundedness and/or compactness assumptions that are not present in our results (see for instance condition (5) in Section 1.3.4 of \cite{Pil99}). In particular, our work represents a nontrivial extension of these results. Moreover, our unified approach gives us several types of shadowing at once (see Remark \ref{rem: unified app}).

\section{Preliminaries}
\subsection{Banach sequence spaces}
In this subsection we recall some basic
definitions and properties from the theory of Banach sequence spaces. The material is taken from~\cite{DD, Sasu} where the reader can also find more details. 

Let $\mathcal{S}(\Z)$ be the set of all sequences $\mathbf{s}=(s_n)_{n\in \Z}$ of real numbers. We say that a linear subspace $B\subset \mathcal{S}(\Z)$ is a \emph{normed sequence space} (over $\Z$) if there exists a norm $\lVert \cdot \rVert_B \colon B \to \R_0^+$ such that if $\mathbf{s}'=(s_n')_{n\in \Z}\in B$ and $\lvert s_n\rvert \le \lvert s_n'\rvert$ for $n\in \Z$, then $\mathbf{s}=(s_n)_{n\in \Z}\in B$ and $\lVert \mathbf{s}\rVert_B \le \lVert \mathbf{s}'\rVert_B$. If in addition $(B, \lVert \cdot \rVert_B)$ is complete, we say that $B$ is a \emph{Banach sequence space}.

Let $B$ be a Banach sequence space over $\Z$. We say that $B$ is \emph{admissible} if:
\begin{enumerate}
\item
$\chi_{\{n\}} \in B$ and $\lVert \chi_{\{n\}}\rVert_B >0$ for $n\in \Z$, where $\chi_A$ denotes the characteristic function of the set $A\subset \Z$;
\item
for each $\mathbf{s}=(s_n)_{n\in \Z}\in B$ and $m\in \Z$, the sequence $\mathbf{s}^m=(s_n^m)_{n\in \Z}$ defined by $s_n^m=s_{n+m}$ belongs to $B$ and  $\lVert \mathbf{s}^m \rVert_B = \lVert \mathbf{s}\rVert_B$.
\end{enumerate}
Note that it follows from the definition that for each admissible Banach space $B$ over $\Z$, we have that $\lVert \chi_{\{n\}}\rVert_B=\lVert \chi_{\{0\}}\rVert_B$ for each $n\in \Z$. Throughout this paper we will assume for the sake of simplicity  that $\lVert \chi_{\{0\}}\rVert_B=1$.

We recall some  explicit examples of admissible  Banach sequence spaces over $\Z$ (see~\cite{DD,Sasu}).

\begin{example}\label{ex1}
The set $l^\infty =\{ \mathbf{s}=(s_n)_{n\in \Z} \in \mathcal{S} (\Z): \sup_{n\in \Z} \lvert s_n \rvert < \infty \}$ is an admissible Banach sequence space when equipped with the norm $\lVert \mathbf{s} \rVert =\sup_{n\in \Z} \lvert s_n \rvert$.
\end{example}

\begin{example}\label{ex2}
The set $c_0=\{ \mathbf{s}=(s_n)_{n\in \Z} \in \mathcal{S} (\Z): \lim_{\lvert n\rvert \to \infty} \lvert s_n\rvert=0\}$ is an admissible Banach sequence space when equipped with the norm $\lVert \cdot \rVert$ from the previous example. 
\end{example}

\begin{example}\label{ex3}
For each $p\in [1, \infty )$, the set \[ l^p= \bigg{\{} \mathbf{s}=(s_n)_{n\in \Z} \in \mathcal{S}(\Z): \sum_{n\in \Z} \lvert s_n \rvert^p <\infty \bigg{\}} \] is an admissible Banach sequence space when equipped with the norm \[\lVert \mathbf{s} \rVert= \bigg{(}\sum_{n\in \Z} \lvert s_n \rvert^p \bigg{)}^{1/p}.\]
\end{example}

\begin{example}[Orlicz sequence spaces]
Let $\phi \colon (0, +\infty) \to (0, +\infty]$ be a nondecreasing nonconstant left-continuous function. We set $\psi(t)=\int_0^t \phi(s) \, ds$  for $t\ge 0$. Moreover, for each $\mathbf s=(s_n)_{n\in \Z} \in \mathcal S(\Z)$, let $M_\phi (\mathbf s)=\sum_{n\in \Z} \psi(\lvert s_n \rvert)$. Then
\[
B=\bigl\{ \mathbf s \in \mathcal S(\Z) : M_\phi (c \mathbf s)<+\infty \ \text{for some} \ c>0 \bigr\}
\]
is an admissible Banach sequence space when equipped with the norm
\[
\lVert \mathbf s \rVert=\inf \bigl\{ c>0 : M_\phi ( \mathbf s / c ) \le 1 \bigr\}.
\]
\end{example}

We will also need the following auxiliary result (see~\cite[Lemma 2.3.]{Sasu}).
\begin{proposition}\label{bounds}
Let $B$ be an admissible Banach sequence space. For $\mathbf s=(s_n)_{n\in \Z}\in B$ and $\lambda >0$, we define sequences $\mathbf s^i=(s_n^i)_{n\in \Z}$, $i=1, 2$ by 
\[
s_n^1:=\sum_{m\ge 0}e^{-\lambda m}s_{n-m} \quad \text{and} \quad s_n^2:=\sum_{m\ge 1}e^{-\lambda m}s_{n+m},
\]
for $n\in \Z$.  Then, $\mathbf s^1, \mathbf s^2\in B$ and in addition, 
\[
\lVert \mathbf s^1\rVert_B \le \frac{1}{1-e^{-\lambda}}\lVert \mathbf s\rVert_B \quad \text{and} \quad \lVert \mathbf s^2\rVert_B \le \frac{e^{-\lambda} }{1-e^{-\lambda}}\lVert \mathbf s\rVert_B.
\]
\end{proposition}

\subsection{Banach spaces associated to Banach sequence spaces}
Let us now introduce sequence spaces that will play important role in our arguments. 
Let $X$ be an arbitrary Banach space and $B$ any Banach sequence space over $\Z$ with norm $\lVert \cdot \rVert_B$. Set
\[
X_B:=\bigg{\{} \mathbf x=(x_n)_{n\in \Z} \subset X: (\lVert x_n\rVert)_{n\in \Z}\in B \bigg{\}}.
\]
Finally, for $\mathbf x=(x_n)_{n\in \Z} \in X_B$ we define 
\begin{equation}\label{nn}
\lVert \mathbf x\rVert_B:=\lVert (\lVert x_n\rVert)_{n\in \Z}\rVert_B.
\end{equation}
\begin{remark}
We emphasize that in~\eqref{nn} we slightly abuse the notation since norms on  $B$ and $X_B$ are denoted in the same way. However, this will cause no confusion since in the rest of the paper we will deal with spaces $X_B$.
\end{remark}
\begin{example}
Let $B=l^\infty$ (see Example~\ref{ex1}). Then, 
\[
X_B=\bigg{\{} \mathbf x=(x_n)_{n\in \Z} \subset X: \sup_{n\in \Z} \lVert x_n\rVert<\infty \bigg{\}}.
\]
\end{example}
The proof of the following result is straightforward (see~\cite{DD, Sasu}).
\begin{proposition}
$(X_B, \lVert \cdot \rVert_B)$ is a Banach space. 
\end{proposition}
\subsection{Exponential dichotomy and trichotomy}
In this subsection we recall the crucial concepts of exponential dichotomy and trichotomy. 
Let $I\in \{\Z, \Z_0^+, \Z_0^-\}$
and take  an arbitrary Banach space  $X=(X, \lVert \cdot \rVert)$. Finally, let $(A_m)_{m\in I}$ be a  sequence of bounded linear operators on $X$. For $m, n\in I$ such that $m\ge n$, set
\[
\cA(m,n)=\begin{cases}
A_{m-1}\cdots A_n & \text{if $m>n$,} \\
\Id & \text{if $m=n$.} 
\end{cases}
\]
We say that the sequence $(A_m)_{m\in I}$ admits an \emph{exponential dichotomy} (on $I$)  if:
\begin{enumerate}
\item there exists a sequence $(P_m)_{m\in I}$ of projections on $X$ such that  
\begin{equation}\label{P}
P_{m+1}A_m=A_mP_m
\end{equation}
 for each $m\in I$ such that $m+1\in I$;
\item  $A_m\rvert_{\Ker P_m} \colon \Ker P_m \to \Ker P_{m+1}$ is an invertible operator  for each $m\in I$ such that $m+1\in I$;
\item there exist $C, \lambda >0$ such that for $m, n\in I$, we have 
\begin{equation}\label{ED1}
\lVert \cA(m,n)P_n\rVert \le Ce^{-\lambda (m-n)} \quad \text{if $m\ge n$}
\end{equation}
and
\begin{equation}\label{ED2}
\lVert \cA(m,n)(\Id-P_n)\rVert \le Ce^{-\lambda (n-m)} \quad \text{if $m\le n$,}
\end{equation}
where 
\[
\cA(m, n):=\big{(}\cA(n, m)\rvert_{\Ker P_m} \big{)}^{-1} \colon \Ker P_n \to \Ker P_m,
\]
for $m\le n$. 
\end{enumerate}
We also introduce the notion of an exponential trichotomy. We say that  a sequence $(A_m)_{m\in \Z}$ of bounded linear operators on $X$ admits an \emph{exponential trichotomy} (on $\Z$)  if there exist $C, \lambda >0$ and projections $P_m^i$, $m\in \Z$, $i\in \{1, 2,3\}$ such that:
\begin{enumerate}
\item for $m\in \Z$, $P_m^1+P_m^2+P_m^3=\Id$;
\item for $m\in \Z$ and $i,j \in \{1, 2, 3\}$, $i\neq j$ we have that $P_m^iP_m^j=0$;
\item \[
P_{m+1}^iA_m=A_mP_m^i \quad \text{for $m\in \Z$ and $i\in \{1, 2, 3\}$;}
\]
\item $A_m\rvert_{\Ker P_m^1} \colon \Ker P_m^1 \to \Ker P_{m+1}^1$ is invertible for each $m\in \Z$;
\item for $m\ge n$, 
\begin{equation}\label{5:37}
\lVert \cA(m, n)P_n^1\rVert \le Ce^{-\lambda (m-n)};
\end{equation}
\item for $m\le n$,
\begin{equation}\label{5:38}
\lVert \cA(m,n)P_n^2\rVert \le Ce^{-\lambda (n-m)}, 
\end{equation}
where 
\[
\cA(m,n):=\big{(} \cA(n,m)\rvert_{\Ker P_m^1}\big{)}^{-1}\colon \Ker P_n^1 \to \Ker P_m^1;
\]
\item 
\begin{equation}\label{5:39}
\lVert \cA(m, n)P_n^3\rVert \le Ce^{-\lambda (m-n)} \quad \text{for $m \ge n$,}
\end{equation}
and
\begin{equation}\label{5:40}
\lVert \cA(m, n)P_n^3\rVert \le Ce^{-\lambda (n-m)} \quad \text{for $m \le n$.}
\end{equation}
\end{enumerate}
Obviously the notion of an exponential dichotomy on $\Z$ is a special case of the notion of an exponential trichotomy and corresponds to the case when $P_m^3=0$ for $m\in \Z$.

\begin{remark}
We stress that the notion of an exponential dichotomy was essentially introduced by Perron~\cite{Per30} and plays a central tole in the qualitative theory of nonautonomous systems (see~\cite{Co-78,Henry2}). For the case of infinite-dimensional and noninvertible dynamics with discrete time, the notion of an exponential dichotomy  was first studied by Henry~\cite{Henry2}.

Although extremely useful, the notion of an exponential dichotomy is sometimes restrictive and it is of interest to look for weaker concepts of asymptotic behaviour. The notion of an exponential trichotomy studied in this paper was introduced by Elaydi and Hajek~\cite{EH} (with further contributions  by Papaschinopoulos~\cite{P} and Alonso, Hong and  Obaya~\cite{AHO}) and 
represents one of many possible meaningful and useful  extensions of the notion of an exponential dichotomy. For the  study of a similar but different concept of trichotomy, we refer to~\cite{Pal96, SS1,SS2} and references therein.

\end{remark}

The following result is a  modification of~\cite[Proposition 2.3.]{AHO} or~\cite[Proposition 1.]{P}. More precisely, in contrast to~\cite{AHO,P} we don't restrict to the case when $B=l^\infty$.
\begin{theorem}\label{adm}
Assume that a sequence $(A_m)_{m\in \Z}$ admits an exponential trichotomy and let $B$ be an arbitrary admissible Banach sequence space. Then, there exists a bounded operator $G \colon X_B \to X_B$ such that for $\mathbf x=(x_n)_{n\in \Z}, \mathbf y=(y_n)_{n\in \Z}\in X_B$, the following assertions are equivalent:
\begin{enumerate}
\item $G\mathbf y=\mathbf x$;
\item for each $n\in \Z$,
\begin{equation}\label{ae}
x_{n+1}-A_n x_n=y_{n+1}.
\end{equation}
\end{enumerate}
\end{theorem}

\begin{proof}
For each $n, m\in \Z$, set 
\[
\mathcal G(n,m):=\begin{cases}
\cA(n,m)P_m^1 & \text{if $m\le 0\le n$ or $m\le n\le 0$;}\\
-\cA(n,m)(\Id-P_m^1) & \text{if $n<m\le 0$;}\\
\cA(n,m)(\Id-P_m^2) & \text{if $0<m\le n$;} \\
-\cA(n,m)P_m^2 & \text{if $0\le n<m$ or $n\le 0<m$.}
\end{cases}
\]
Observe that it follows readily from~\eqref{5:37}, \eqref{5:38}, \eqref{5:39} and~\eqref{5:40} that
\begin{equation}\label{green}
\lVert \mathcal  G(n,m)\rVert \le 2Ce^{-\lambda \lvert m-n\rvert} \quad \text{for $m, n\in \Z$.}
\end{equation}
For $\mathbf y=(y_n)_{n\in \Z}$ and $n\in \Z$, let 
\[
(G\mathbf y)_n:=\sum_{m\in \Z}\mathcal G(n,m+1)y_{m+1}.
\]
Observe that~\eqref{green} implies that
\[
\begin{split}
\lVert (G\mathbf y)_n \rVert &\le \sum_{m=-\infty}^{n-1}\lVert \mathcal G(n,m+1)y_{m+1}\rVert +\sum_{m=n}^\infty \lVert \mathcal G(n,m+1)y_{m+1}\rVert  \\
&\le 2C\sum_{m=-\infty}^{n-1} e^{-\lambda (n-m-1)}\lVert y_{m+1}\rVert +2C \sum_{m=n}^\infty e^{-\lambda (m+1-n)}\lVert y_{m+1}\rVert,
\end{split}
\]
for $n\in \Z$. Hence, it follows from Proposition~\ref{bounds} that $G\mathbf y\in Y_B$ and 
\begin{equation}\label{7:55}
\lVert G\mathbf y\rVert_B \le 2C\frac{1+e^{-\lambda}}{1-e^{-\lambda}}\lVert \mathbf y\rVert_B. 
\end{equation}
Finally, in~\cite[Proposition 1.]{P} it is proved that $\mathbf x=G\mathbf y$ satisfies~\eqref{ae}.
\end{proof}
In the case of exponential dichotomy we can say more. More precisely, we have the following result established in~\cite[Theorem 3.5.]{Sasu}.
\begin{theorem}\label{adm2}
Assume that a sequence $(A_m)_{m\in \Z}$ admits an exponential dichotomy  and let $B$ be an arbitrary admissible Banach sequence space. Then, for each $\mathbf y=(y_n)_{n\in \Z}\in X_B$ there exists a unique $\mathbf x=(x_n)_{n\in \Z}\in X_B$ such that~\eqref{ae} holds.
Furthermore, $\mathbf x=G\mathbf y$, where $G$ is as in the statement of Theorem~\ref{adm}.
\end{theorem}

\section{Main result}\label{MR}
\subsection{Setup} \label{sec: setup}
Let $B$ be an admissible Banach sequence space,  $X$  a  Banach space and $(A_m)_{m\in \Z}$ a sequence of  bounded linear operators on $X$ that admits an exponential trichotomy. 
Furthermore, let $f_n \colon X\to X$, $n\in \Z$ be a sequence of  maps such that there exists $c>0$  so that
\begin{equation}\label{fg}
\lVert f_n(x)-f_n(y)\rVert \le c\lVert x-y\rVert,
\end{equation}
for each $n\in \Z$ and $x, y\in X$.

We consider a nonautonomous and nonlinear dynamics defined by the equation
\begin{equation}\label{nnd}
x_{n+1}=F_n(x_n),  \quad n\in \Z,
\end{equation}
where \[F_n:=A_n+f_n.\]

Let us now recall some notation introduced in~\cite{BD19}.
Given $\delta >0$, the sequence $(y_n)_{n\in \Z} \subset X$ is said to be an $(\delta, B)$-\emph{pseudotrajectory} for~\eqref{nnd} if  
$(y_{n+1}-F_n(y_n))_{n\in \Z}\in X_B$ and 
\begin{equation}\label{pseudo}
\lVert (y_{n+1}-F_n(y_n))_{n\in \Z} \rVert_B \le \delta. 
\end{equation}
\begin{remark}
When $B=l^\infty$ (see Example~\ref{ex1}), condition~\eqref{pseudo} reduces to 
\[
\sup_{n\in \Z}  \lVert y_{n+1}-F_n(y_n) \rVert \le \delta.
\]
The above requirement represents a usual definition of a pseudotrajectory in the context of smooth dynamics (see~\cite{Pal00, Pil99}). 
\end{remark}

We say that~\eqref{nnd} has a \emph{$B$-shadowing property} if for every $\varepsilon>0$ there exists $\delta >0$ so that for every $(\delta, B)$-pseudotrajectory $(y_n)_{n\in \Z}$, there exists a sequence $(x_n)_{n\in \Z}$ satisfying \eqref{nnd}  and such that 
$(x_n-y_n)_{n\in \Z} \in X_B$ together with
\begin{equation}\label{wv}\lVert (x_n-y_n)_{n\in \Z}\rVert_B \le \varepsilon.\end{equation}
Moreover, if there exists $L>0$ such that $\delta$ can be chosen as $\delta=L\epsilon$, we say that~\eqref{nnd} has the \emph{$B$-Lipschitz shadowing property}.

Let $G \colon X_B \to X_B$ be a linear operator given by Theorem~\ref{adm}.
\begin{theorem}\label{cor: uniqueness}
Assume that 
\begin{equation}\label{cG}
c\lVert G\rVert<1. 
\end{equation}
Then, the  system~\eqref{nnd} has the  $B$-Lipschitz shadowing property. 
\end{theorem}

\begin{proof}
Take an arbitrary $\epsilon >0$ and let
\begin{equation}\label{k}
K:=\frac{\lVert G\rVert}{1-c\lVert G\rVert}.
\end{equation}
Finally, set $\delta:=\frac{\epsilon}{K}>0$ and take an arbitrary  $(\delta, B)$-pseudotrajectory $\mathbf y=(y_n)_{n\in \Z}\subset X$ of~\eqref{nnd}.  For $n\in \Z$, we define $g_n \colon X\to X$ by 
\[
g_n(v)=f_n(y_n+v)-f_n(y_n)+F_n(y_n)-y_{n+1}, \quad v\in X. 
\]
Furthermore, for $\mathbf x=(x_n)_{n\in \Z}\in X_B$, let $S(\mathbf x)$ be the  sequence defined by 
\[
(S(\mathbf x))_n=g_{n-1}(x_{n-1}), \quad n\in \Z.
\]
Observe that it follows from~\eqref{fg} and~\eqref{pseudo} that  $S (\mathbf x)\in X_B$. Finally, set 
\[
T(\mathbf x)=GS(\mathbf x).
\]
We claim that $T$ is a contraction on 
\[
D(\mathbf 0, \epsilon):=\{\mathbf x\in X_B: \lVert \mathbf x\rVert_B\le \epsilon \}.
\]
Indeed, let us choose $\mathbf x^1=(x_n^1)_{n\in \Z}$ and $\mathbf x^2=(x_n^2)_{n\in \Z}$ that belong to $D(\mathbf 0, \epsilon)$. Observe that it follows from~\eqref{fg} that
\[
\begin{split}
\lVert g_n(x_n^1)-g_n(x_n^2)\rVert &=\lVert f_n(y_n+x_n^1)-f_n(y_n+x_n^2)\rVert  \\
&\le c\lVert x_n^1-x_n^2\rVert,
\end{split}
\]
for $n\in \Z$. Hence, 
\[
\lVert S(\mathbf x^1)-S(\mathbf x^2)\rVert_B \le c\lVert \mathbf x^1-\mathbf x^2\rVert_B.
\]
Consequently, 
\[
\lVert T(\mathbf x^1)-T(\mathbf x^2)\rVert_B \le \lVert G\rVert \cdot \lVert S(\mathbf x^1)-S(\mathbf x^2)\rVert_B \le c\lVert G\rVert \cdot  \lVert \mathbf x^1-\mathbf x^2\rVert_B.
\]
Hence, \eqref{cG}  implies  that $T$ is a contraction on $D(\mathbf 0, \epsilon)$.

We now show that $T$ maps $D(\mathbf 0, \epsilon)$ into itself. Take an arbitrary  $\mathbf x\in D(\mathbf 0, \epsilon)$. We have that 
\[
\begin{split}
\lVert T(\mathbf x)\rVert_B &\le \lVert T(\mathbf 0)\rVert_B+\lVert T(\mathbf x)-T(\mathbf 0)\rVert_B \\
&\le \lVert G\rVert \cdot  \lVert S(\mathbf 0)\rVert_B+c\lVert G\rVert \cdot \lVert \mathbf x\rVert_B \\
&\le \lVert G\rVert \cdot  \lVert  S(\mathbf 0)\rVert_B+ \epsilon c\lVert G\rVert.
\end{split}
\]
Since $\mathbf y=(y_n)_{n\in \Z}$ is an $(\delta, B)$-pseudotrajectory, we have that $\lVert S(\mathbf 0)\rVert_B\le \delta=\frac{\epsilon}{K}$ and consequently
\[
\lVert T(\mathbf x)\rVert_B \le \epsilon \bigg{(}\frac{\lVert G\rVert}{K}+c\lVert G\rVert \bigg{)}=\epsilon,
\]
where in the last equality we used~\eqref{k}.

We conclude  that $T$ has a fixed point $\mathbf x=(x_n)_{n\in \Z} \in D(\mathbf 0, \epsilon)$. Hence, $\mathbf x=GS(\mathbf x)$. In a view of Theorem~\ref{adm}, we deduce that 
\[
x_n=A_{n-1}x_{n-1}+g_{n-1}(x_{n-1}) \quad \text{for $n\in \Z$.}
\]
Therefore, $\mathbf x+\mathbf y=(x_n+y_n)_{n\in \Z}$ is a solution of~\eqref{nnd} and
\[
\lVert \mathbf x+\mathbf y-\mathbf y\rVert_B=\lVert \mathbf x\rVert_B\le \epsilon.
\]
This completes the proof of the theorem. 
\end{proof}
\begin{remark}
Observe that it follows from~\eqref{7:55} that~\eqref{cG} holds for any $c$ such that
\[
0<c<\frac{1-e^{-\lambda}}{2C(1+e^{-\lambda})},
\]
where $C, \lambda>0$ are the constants associated with the trichotomy of $(A_m)_{m\in \Z}$.
\end{remark}

\begin{remark}
Let us now briefly describe the relationship between Theorem~\ref{cor: uniqueness} and  the results dealing with the shadowing of structurally stable diffeomorphisms discussed in~\cite[Chapter 2]{Pil99}. Let $M$ be a compact Riemannian manifold and let $f\colon M \to M$ be a $C^1$-structurally stable diffeomorphism.
We recall that this means that there exists $\delta >0$ with the property that for each $C^1$-diffeomorphism $g$ on $M$ such that $d_{C^1}(f, g)<\delta$, there exists a homeomorphism $h\colon M\to M$ satisfying $h\circ f=g\circ h$. Here, $d_{C^1}(f,g)$ is given by
\[
d_{C^1}(f,g):=\sup_{p\in M}d(f(p), g(p))+\sup_{p\in M}\lVert Df(p)-Dg(p)\rVert,
\]
where $d$ is a distance on $M$ induced by the Riemannian metric.  Furthermore, we recall (see~\cite[Theorem 2.2.4]{Pil99}) that a $C^1$-diffeomorphism on $M$ is structurally stable if and only if the following two conditions hold:
\begin{itemize}
\item $f$ is uniformly hyperbolic on the set $\Omega(f)$ of its nonwandering points and the set of periodic points of $f$ is dense in $\Omega(f)$;
\item for every $p, q\in \Omega (f)$,  the stable manifold $W^s(p)$  and the unstable manifold $W^u(q)$  are transverse.
\end{itemize}

Take now a structurally stable $C^1$-diffeomorphism $f$ on $M$. It is known (see~\cite[Lemma 2.2.16]{Pil99}) that for each $p\in M$ there are subspaces $S(p)$ and $U(p)$ of the tangent space $T_pM$ with the following properties:
\begin{itemize}
\item $T_pM=S(p)\oplus U(p)$;
\item $Df(p)S(p)\subset S(f(p))$ and $Df^{-1}(p)U(p)\subset U(f^{-1}(p))$;
\item there exist $C, \lambda >0$ with the property that 
\begin{equation}\label{es1}
\lVert D f^k(p)\Pi^s(p) v\rVert \le Ce^{-\lambda k}\lVert v\rVert \quad \text{for $v\in T_pM$ and $k\ge 0$}
\end{equation}
and
\begin{equation}\label{es2}
\lVert D f^{-k}(p)\Pi^u(p)v\rVert  \le Ce^{-\lambda k}\lVert v\rVert \quad \text{for $v\in T_pM$ and $k\ge 0$,}
\end{equation}
where $\Pi^s(p)$ is a projection onto $S(p)$ and $\Pi^u(p)$ is a projection onto $U(p)$.
\end{itemize}
In other words, $f$ possesses a structure similar to that of an Anosov diffeomorphism with the only difference being that $S(p)$ and $U(p)$ are not invariant under the action of $Df$. 
Using these and some additional properties, one can show (see~\cite[Theorem 2.2.7]{Pil99}) that $f$ has the Lipschitz shadowing property. In addition, Pilyugin and Tikhomirov~\cite{PT} proved that the converse statement also holds. More precisely, if $f$ is a $C^1$-diffeomorphism with the Lipschitz shadowing property, then $f$ is structurally stable. 

Let us now comment on how these results relate to ours. We continue to consider a structurally stable $C^1$-diffeomorphism $f$ on $M$. In addition, assume that the tangent bundle $TM$ is isomorphic to $\R^k$, where $k=\dim M$. 
Take $p\in M$ and let
\[
A_n=Df (f^n(p)), \quad n\in \Z.
\]
We claim that the sequence $(A_n)_{n\in \Z}$ admits an exponential trichotomy. Indeed, take an arbitrary $\mathbf y=(y_n)_{n\in \Z}\subset \R^k$ such that $\lVert \mathbf y\rVert_\infty:=\sup_{n\in \Z}\lVert y_n\rVert<\infty$. For $n\in \Z$, set
\[
\begin{split}
x_n &=\sum_{k=0}^\infty Df^k (f^{n-k}(p))\Pi^s (f^{n-k}(p))y_{n-k} \\
&\phantom{=}-\sum_{k=1}^\infty Df^{-k}(f^{n+k} (p))\Pi^u(f^{n+k}(p))y_{n+k}.
\end{split}
\]
By~\eqref{es1} and~\eqref{es2}, we have that
\[
\lVert \mathbf x\rVert_\infty=\sup_{n\in \Z}\lVert x_n\rVert \le C\frac{1+e^{-\lambda}}{1-e^{-\lambda}}\lVert \mathbf y\rVert_\infty, 
\]
where $\mathbf x=(x_n)_{n\in \Z}$. In addition, it is easy to verify that 
\[
x_{n+1}-A_nx_n=y_{n+1}, \quad \text{for $n\in \Z$.}
\]
Hence, it follows from~\cite[Proposition 1]{P} that $(A_n)_{n\in \Z}$ admits an exponential trichotomy.  By taking into account that we can regard $f$ as the perturbation of $Df$, we observe that this setting is similar to the one studied in the present paper. However, there are important differences. Namely, in order
to study the shadowing property for a  dynamics acting  on a  noncompact phase space, we first start with a linear dynamics that admits an exponential trichotomy and then introduce a suitable class of nonlinear perturbations which exhibit the shadowing property. This class is completely determined by the constants in the notion of exponential trichotomy (associated to the linear part). 
Besides this important difference,  in contrast to the above mentioned results from~\cite{Pil99}, 
in the present paper  we consider the general case of  nonautonomous and  noninvertible dynamics that acts on an arbitrary Banach space.
\end{remark}
Our results in particular  apply to the case of linear dynamics
\begin{equation}\label{ld}
x_{n+1}=A_n x_n, \quad n\in \Z.
\end{equation}

\begin{corollary}
System~\eqref{ld} has the $B$-Lipschitz shadowing property, for any admissible Banach sequence space $B$. 
\end{corollary}

\begin{proof}
The desired conclusion follows by applying Theorem~\ref{cor: uniqueness} in the particular case when $f_n=0$, $n\in \Z$.
\end{proof}

Now we obtain  a partial converse to the previous corollary.
\begin{proposition}\label{257}
Assume that $X$ is finite-dimensional and that $(A_m)_{m\in \Z}$ is a sequence of   linear operators on $X$ such that~\eqref{ld} has the $l^\infty$-shadowing. Furthermore, suppose that $A_m$ is invertible for each $m\in \Z$.
 Then, $(A_m)_{n\in \Z}$ admits an exponential trichotomy.
\end{proposition}

\begin{proof}
Choose  $\delta>0$ that corresponds to $\varepsilon=1$ in the notion of $l^\infty$-shadowing. 
We will prove that for every $\mathbf z=(z_n)_{n\in \Z}\in X_{l^\infty}$ there exists $\mathbf w=(w_n)_{n\in \Z}\in X_{l^\infty}$ such that 
\begin{equation}\label{iu}
w_{n+1}-A_nw_n=z_{n+1}, \quad n\in \Z.
\end{equation}
Choose a sequence  $\mathbf y=(y_n)_{n\in \Z}\subset X$ (which is completely determined with $y_0$)  such that
\[
y_{n+1}=A_n y_n+\frac{\delta}{\lVert \mathbf z\rVert_{l^\infty}}z_{n+1}, \quad n\in \Z.
\]
Then, $\mathbf y$ is an $(\delta, l^\infty)$-pseudotrajectory. Hence, there exists a solution $\mathbf x=(x_n)_{n\in \Z}$  of~\eqref{ld} such that $\sup_{n\in \Z}\lVert x_n-y_n\rVert \le 1$.
Set
\[
w_n=\frac{\lVert \mathbf z\rVert_{l^\infty}}{\delta}(y_n-x_n)\quad \text{for $n\in \Z$.}
\]
Obviously, $\mathbf w=(w_n)_{n\in \Z}\in X_{l^\infty}$ and it is easy to verify that~\eqref{iu} holds.  The conclusion of the proposition now follows directly from~\cite[Proposition 1.]{P}.
\end{proof}

\begin{remark}
We observe that Proposition~\ref{257}  is false (in general) for infinite-dimensional dynamics even in the autonomous case when all $A_m$ coincide (see for  example~\cite[Remark 9(a)]{BCDMP18}).
\end{remark}

It turns out that in the case when $(A_m)_{m\in \Z}$ admits an exponential dichotomy, we can say more. We first recall (see Theorem~\ref{adm2}) that in that case, for each $\mathbf y\in X_B$,  $\mathbf x:=G\mathbf y$ is the  unique sequence in $X_B$ such that~\eqref{ae} holds. 
\begin{theorem} \label{theo: dichotomy uniqueness}
Assume that the sequence $(A_m)_{m\in \Z}$ admits an exponential dichotomy and that~\eqref{cG} holds. Then, \eqref{nnd} has an $B$-Lipschitz shadowing property.  Furthermore, trajectory that shadows each pseudotrajectory is unique. 
\end{theorem}

\begin{proof}
We will use the same notation as in the proof of Theorem~\ref{cor: uniqueness}.
In a view of  Theorem~\ref{cor: uniqueness}, we only need to establish the  uniqueness part.  Let $\mathbf y$ be an $(\delta, B)$-pseudotrajectory for~\eqref{nnd} and assume that $\mathbf z^1=(z_n^1)_{n\in \Z}, \mathbf z^2=(z_n^2)_{n\in \Z}$ are trajectories of~\eqref{nnd} such that
\[
\lVert \mathbf z^i-\mathbf y\rVert_B \le \epsilon \quad \text{for $i=1, 2$.}
\]
Then, 
\[
z_n^i-y_n=A_{n-1}(z_{n-1}^i-y_{n-1})+g_{n-1}(z_{n-1}^i-y_{n-1}),
\]
for $n\in \Z$ and $i\in \{1, 2\}$. Consequently, 
\begin{equation}\label{bs}
z_n^1-z_n^2=A_{n-1}(z_{n-1}^1-z_{n-1}^2)+w_n, 
\end{equation}
where
\[
w_n:=g_{n-1}(z_{n-1}^1-y_{n-1})-g_{n-1}(z_{n-1}^2-y_{n-1}), \quad n\in \Z.
\]
Let $\mathbf w=(w_n)_{n\in \Z}$. It follows from~\eqref{fg} that 
\[
\lVert \mathbf w\rVert_B \le c\lVert \mathbf z^1-\mathbf z^2\rVert_B.
\]
On the other hand, \eqref{bs} implies that 
\[
\lVert \mathbf z^1-\mathbf z^2\rVert_B\le \lVert G\rVert \cdot  \lVert \mathbf w\rVert_B.
\]
By combining the last two inequalities, we conclude that
\[
\lVert \mathbf z^1-\mathbf z^2\rVert_B \le c \lVert G\rVert \cdot \lVert \mathbf z^1-\mathbf z^2\rVert_B.
\]
By~\eqref{cG}, we conclude that $\lVert \mathbf z^1-\mathbf z^2\rVert_B=0$ and thus $\mathbf z^1=\mathbf z^2$. The proof of the theorem is completed. 
\end{proof}

\begin{remark} \label{rem: unified app}
Observe that our unified approach gives us all the usual types of shadowing simply by considering different types of admissible Banach sequence spaces $B$. For instance, for $B=l^\infty$ we get the usual notion of Lipschitz shadowing. For $B=l^p$ as in Example \ref{ex3} we get the notion of $l^p$-shadowing and so on.
\end{remark}

\section{One-sided dynamics}
Let us now consider the case of one-sided dynamics on $\Z_0^+$. We stress that the dynamics on $\Z_0^-$ can be treated analogously.

For $\mathbf x=(x_n)_{n\ge 0}\subset X$, we define $\bar{\mathbf x}=(\bar{x}_n)_{n\in \Z}\subset X$  by
\[
\bar{x}_n:=\begin{cases}
x_n & \text{if $n\ge 0$;}\\
0 & \text{if $n<0$.}
\end{cases}
\]
For an admissible Banach sequence space $B$, let 
\[
X_B^+:= \bigg{\{} \mathbf x=(x_n)_{n\ge 0}\subset X: \bar{\mathbf x} \in X_B \bigg{\}}.
\]
Then, $X_B^+$ is the Banach space with respect to the norm $\lVert \mathbf x\rVert_B^+ :=\lVert \bar{\mathbf x}\rVert_B$.

Assume that $(A_m)_{m\ge 0}$ is a sequence of bounded linear operators on $X$ and let $f_n \colon X\to X$, $n\ge 0$ be the sequence of maps such that~\eqref{fg} holds for $n\ge 0$ (and with some $c>0$). 
We consider the associated nonlinear dynamics
\begin{equation}\label{nd2}
x_{n+1}=F_n(x_n) \quad n\ge 0,
\end{equation}
where $F_n:=A_n+f_n$.
Given $\delta >0$, the sequence $(y_n)_{n\ge 0} \subset X$ is said to be an $(\delta, B)$-\emph{pseudotrajectory} for~\eqref{nd2} if  
$(y_{n+1}-F_n(y_n))_{n\ge 0}\in X_B^+$ and 
\begin{equation}\label{pseudo2}
\lVert (y_{n+1}-F_n(y_n))_{n\ge 0} \rVert_B^+ \le \delta. 
\end{equation}
We say that~\eqref{nd2} has an \emph{$B$-shadowing property} if for every $\varepsilon>0$ there exists $\delta >0$ so that for every $(\delta, B)$-pseudotrajectory $(y_n)_{n\ge 0}$, there exists a sequence $(x_n)_{n\ge 0}$ satisfying \eqref{nd2}  and such that 
$(x_n-y_n)_{n\ge 0} \in X_B^+$ together with
\begin{equation}\label{wv2}\lVert (x_n-y_n)_{n\ge 0}\rVert_B^+ \le \varepsilon.\end{equation}
\begin{theorem}\label{248}
Assume that $(A_m)_{m\ge 0}$ admits an exponential dichotomy and let $B$ be an admissible Banach sequence space. Then,  if  $c>0$ is sufficiently small~\eqref{nd2} has a $B$-shadowing property.
\end{theorem}

\begin{proof}
We extend the sequence $(A_m)_{m\ge 0}$ to a sequence over $\Z$ in the following manner: choose an invertible, hyperbolic linear operator $A$ on $X$ such that $\Ker P_0$ coincides with the unstable subspace of $A$ and let $A_m:=A$ for $m<0$. Then (see~\cite[Lemma 1.]{CP}), $(A_m)_{m\in \Z}$ admits an exponential dichotomy. Consider $G$  as in Theorems~\ref{adm} and~\ref{adm2} and 
let $c>0$ be such that $c\lVert G\rVert<1$. 
 Finally, set 
$f_n=0$ for $n<0$ and consider the nonlinear system
\begin{equation}\label{nd3}
x_{n+1}=F_n(x_n) \quad n\in \Z,
\end{equation}
where $F_n=A_n+f_n=A$ for $n<0$.  Take an arbitrary $\varepsilon>0$, define $K$ as in~\eqref{k} and let $\delta:=\frac{\epsilon}{K}>0$. Furthermore, choose  an $(\delta, B)$-\emph{pseudotrajectory} $\mathbf y=(y_n)_{n\ge 0}$  for~\eqref{nd2}.
We consider $\hat{\mathbf y}=(\hat{y}_n)_{n\in \Z}\subset X$ defined by
\[
\hat{y}_n:=\begin{cases}
y_n & \text{if $n\ge 0$;} \\
A^n y_0 & \text{if $n<0$.}
\end{cases}
\]
Clearly, $\hat{\mathbf y}$ is an $(\delta, B)$-\emph{pseudotrajectory} for~\eqref{nd3}. Hence, it follows from the proof of Theorem~\ref{cor: uniqueness} that there exists a sequence $\mathbf x=(x_n)_{n\in \Z}\subset X$ that solves~\eqref{nd3} and such that
$\lVert \mathbf x-\hat{\mathbf y}\rVert_B \le \epsilon$. Then, $\mathbf z=(x_n)_{n\ge 0}$ is a solution of~\eqref{nd2} such that $\lVert \mathbf z-\mathbf y\rVert_B^+\le \epsilon$ and the proof is complete. 
\end{proof}

As in the case of two-sided dynamics, our results in particular apply to the case of linear dynamics
\begin{equation}\label{ld2}
x_{n+1}=A_n x_n, \quad n\ge 0.
\end{equation}

\begin{corollary}\label{957}
Assume that $(A_m)_{m\ge 0}$ admits an exponential dichotomy and let $B$ be an admissible Banach sequence space. Then, \eqref{ld2} has an $B$-Lipschitz shadowing property. 
\end{corollary}

\begin{proof}
The desired conclusion follows directly from Theorem~\ref{248} applied to the case when $f_n=0$ for $n\ge 0$.
\end{proof}
We also have the following partial converse to Corollary~\ref{957}. 
\begin{proposition}\label{prop: converse one sided}
Assume that $X$ is finite-dimensional and that~\eqref{ld2} has an $l^\infty$-Lipschitz shadowing property. Then, $(A_m)_{m\ge 0}$ admits an exponential dichotomy. 
\end{proposition}

\begin{proof}
By proceeding as in the proof of Proposition~\ref{257}, one can show that for every $\mathbf z=(z_n)_{n\ge 0}\in X_B^+$ such that $z_0=0$, there exists $\mathbf w=(w_n)_{n\ge 0}\in X_B^+$ satisfying
\[
w_{n+1}-A_n w_n=z_{n+1} \quad \text{for $n\ge 0$.}
\]
Hence, \cite[Theorem 3.2.]{HM} implies that $(A_m)_{m\ge 0}$ admits an exponential dichotomy. 
\end{proof}

\section{A case of continuous time}
In this section we will apply our previous results in order to develop shadowing theory for continuous time dynamics. 

For the sake of simplicity, in this section we  will study only the classical $l^\infty$-shadowing.  We consider a nonlinear differential equation
\begin{equation}\label{ndc}
x'=A(t)x+f(t,x),
\end{equation}
where $A$ is a continuous map from $\R$ to the space of all bounded linear operators  on $X$ satisfying
\[
N:=\sup_{t\in \R} \lVert A(t)\rVert<\infty,
\]
 and $f\colon \R \times X\to X$ is a continuous map. We assume that $f(\cdot,0)=0$  and that there exists $c>0$ such that
\begin{equation}\label{fgc}
\lVert f(t,x)-f(t,y)\rVert \le c\lVert x-y\rVert \quad \text{for $t\in \R$ and $x, y\in X$.}
\end{equation}
 We consider the associated linear equation
\begin{equation}\label{ldc}
x'=A(t)x.
\end{equation}
Let $T(t,s)$ be the (linear) evolution family associated to~\eqref{ldc}. We will suppose that it admits an exponential trichotomy, i.e. that there exists a family of projections $P^i(s)$, $s\in \R$, $i\in \{1, 2, 3\}$ on $X$ and $C, \lambda >0$ such that:
\begin{enumerate}
\item for $s\in \R$, $P^1(s)+P^2(s)+P^3(s)=\Id$;
\item for $s\in \R$, $i, j\in \{1,2,3\}$, $i\neq j$ we have that $P^i(s)P^j(s)=0$;
\item for $t,s\in \R$ and $i\in \{1,2,3\}$,
\begin{equation}\label{proc}
T(t,s)P^i(s)=P^i(t)T(t,s);
\end{equation}
\item for $t\ge s$,
\begin{equation}\label{ed1c}
\lVert T(t,s)P^1(s)\rVert \le Ce^{-\lambda (t-s)};
\end{equation}
\item for $t\le s$,
\begin{equation}\label{ed2c}
\lVert T(t,s)P^2(s)\rVert \le Ce^{-\lambda (s-t)};
\end{equation}
\item 
\begin{equation}\label{ed3c}
\lVert T(t,s)P^3(s)\rVert \le Ce^{-\lambda (t-s)} \quad \text{for $t\ge s$,}
\end{equation}
and
\begin{equation}\label{ed4c}
\lVert T(t,s)P^3(s)\rVert \le Ce^{-\lambda (s-t)} \quad \text{for $t\le s$.}
\end{equation}
\end{enumerate}
Since we assumed that $N<\infty$,  it is easy to show using Gronwall's lemma that there exist $D, b>0$ such that
\begin{equation}\label{ubc}
\lVert T(t,s)\rVert \le De^{b(t-s)} \quad t\ge s.
\end{equation}
Recall that the nonlinear evolution family associated with~\eqref{ndc} is given by
\begin{equation}\label{U}
U(t,s)x=T(t,s)x+\int_s^t T(t,\tau)f(\tau, U(\tau, s)x)\, d\tau,
\end{equation}
for $x\in X$ and $t,s\in \R$. By applying Gronwall's lemma, it is easy to prove that~\eqref{fgc} and~\eqref{ubc} imply that there exist $K, a>0$ such that
\begin{equation}\label{ubc1}
\lVert U(t,s)\rVert \le Ke^{a(t-s)} \quad \text{for $t\ge s$.}
\end{equation}

We now introduce the concept of shadowing in this setting. Let $\delta >0$. A differentiable function $y\colon \R \to X$ is said to be a \emph{$\delta$-pseudotrajectory} for~\eqref{ndc} if 
\[
\sup_{t\in \R} \lVert y'(t)-A(t)y(t)-f(t, y(t))\rVert \le \delta. 
\]
We say that~\eqref{ndc} has the \emph{shadowing property} if for every $\epsilon >0$ there exists $\delta >0$ such that for every $\delta$-pseudotrajectory $y\colon \R \to X$, there exists a solution $x\colon \R \to X$ of~\eqref{ndc} satisfying
\[
\sup_{t\in \R} \lVert x(t)-y(t)\rVert \le \epsilon.
\]
Moreover, if there exists $L>0$ such that $\delta$ can be chosen as $\delta=L\epsilon$, we say that~\eqref{ndc} has the \emph{Lipschitz shadowing property}.
We now formulate and prove the main result of this section. 
\begin{theorem}
If $c>0$ is sufficiently small, then \eqref{ndc} has the Lipschitz shadowing property. 
\end{theorem}

\begin{proof}
Set
\[
A_n:=T(n+1, n), \quad \text{for $n\in \Z$.}
\]
It follows readily from~\eqref{proc}, \eqref{ed1c}, \eqref{ed2c}, \eqref{ed3c} and~\eqref{ed4c}  that $(A_m)_{m\in \Z}$ admits an exponential trichotomy with projections $P_n^i=P^i(n)$, $n\in \Z$, $i\in \{1, 2, 3\}$.
Furthermore, set 
\[
f_n(x)=\int_n^{n+1} T(n+1, \tau)f(\tau, U(\tau, n)x)\, d\tau, \quad \text{for $x\in X$ and $n\in \Z$.}
\]
It follows from~\eqref{fgc}, \eqref{ubc} and~\eqref{ubc1} that
\[
\begin{split}
\lVert f_n(x)-f_n(y)\rVert &\le \int_n^{n+1}\lVert T(n+1, \tau)\rVert \cdot \lVert f(\tau, U(\tau, n)x)-f(\tau, U(\tau, n)y)\rVert \, d\tau \\
&\le c\int_n^{n+1}\lVert T(n+1, \tau)\rVert \cdot \lVert U(\tau, n)(x-y)\rVert\, d\tau \\
&\le c DK e^{a+b}\lVert x-y\rVert, 
\end{split}
\]
an thus there exists $c'>0$ such that
\begin{equation}\label{625}
\lVert f_n(x)-f_n(y)\rVert \le cc' \lVert x-y\rVert \quad \text{for $n\in \Z$ and $x, y\in X$.}
\end{equation}
Set $F_n:=A_n+f_n$ and consider the system
\begin{equation}\label{811}
x_{n+1}=F_n(x_n), \quad n\in \Z. 
\end{equation}
 Observe that it follows from~\eqref{U} that $F_n=U(n+1, n)$ for each $n\in \Z$. 

Since $(A_m)_{m\in \Z}$ admits an exponential trichotomy, it follows from Theorem~\ref{cor: uniqueness} and~\eqref{625} that for sufficiently small $c$, \eqref{811} has the $l^\infty$-Lipschitz shadowing. Let $L>0$ be the constant as in the definition 
of Lipschitz shadowing related to~\eqref{811}. Take $\epsilon >0$ and let 
$\delta:=L'\epsilon$, where
\[
L':=\frac{1}{\bigg{(}1+\frac{e^{N+c}}{L}\bigg{)} e^{N+c}}.
\]
 Furthermore, let $y$ be the $\delta$-pseudotrajectory for~\eqref{ndc}. Then, 
\[
y'=A(t)y+f(t,y)+h,
\]
for some $h\colon \R \to X$ such that $\lVert h(t)\rVert \le \delta$ for $t\in \R$. Take $n\in \Z$ and let $z$ be the solution of~\eqref{ndc} such that $z(n)=y(n)$. Then, for all $t\in [n, n+1]$ we have (see~\eqref{fgc}) that 
\[
\begin{split}
& \lVert y(t)-z(t)\rVert \\
&\le \bigg{\lVert} \int_n^t (A(s)(y(s)-z(s))+f(s, y(s))-f(s, z(s))+h(s))\, ds \bigg{\rVert} \\
&\le \delta+(N+c)\int_n^t \lVert y(s)-z(s)\rVert\, ds.
\end{split}
\]
Hence, it follows from Gronwall's lemma that
\[
\lVert y(t)-z(t)\rVert  \le \delta e^{N+c}, \quad \text{for $t\in [n, n+1]$.}
\]
In particular, 
\[
\lVert y(n+1)-F_n (y(n))\rVert =\lVert y(n+1)-z(n+1)\rVert \le \delta e^{N+c},
\]
for every $n\in \Z$. Hence, the sequence $(y_n)_{n\in \Z}\subset X$ defined by $y_n:=y(n)$ is an ($\delta e^{N+c}$, $l^\infty)$-pseudotrajectory for~\eqref{811}. Hence, there exists a solution $(x_n)_{n\in \Z}$ of~\eqref{811} such that
$\sup_{n\in \Z}\lVert x_n-y_n\rVert \le \frac{\delta e^{N+c}}{L}$. We define $x\colon \R \to X$ by 
\[
x(t)=U(t,n)x_n \quad \text{$n\in \Z$, $t\in [n, n+1)$.}
\]
Then, $x$ is a solution of~\eqref{ndc}. Finally, observe that that for $n\in \Z$ and $t\in [n, n+1)$ we have  that
\[
\begin{split}
& \lVert x(t)-y(t)\rVert \le \lVert x_n-y_n\rVert \\
&+ \bigg{\lVert} \int_n^t (A(s)(x(s)-y(s))+f(s, x(s))-f(s, y(s))-h(s))\, ds \bigg{\rVert} \\
&\le \delta \bigg{(}1+\frac{e^{N+c}}{L}\bigg{)}+(N+c) \int_t^n \lVert x(s)-y(s)\rVert\, ds.
\end{split}
\]
Hence, Gronwall's lemma implies that 
\[
\sup_{t\in \R}\lVert x(t)-y(t)\rVert \le  \delta \bigg{(}1+\frac{e^{N+c}}{L}\bigg{)} e^{N+c}=\epsilon.
\]
\end{proof}

One can now easily formulate and prove  continuous time versions of all other results we established in Section~\ref{MR}. We refrain from doing this since it represents a  very simple exercise and requires only simple modification of the arguments we developed.

\section{Applications}

\subsection{Hyers-Ulam stability}\label{sec: hyers-ulam stab}
It turns out that our results are closely related to the so-called  Hyers-Ulam stability and in fact, can be used to obtain new results related to this concept. We will not attempt to survey various results in the literature regarding the Hyers-Ulam stability but will rather focus on the recent papers~\cite{ BBT2, BRST, BLR} and the results obtained there. 

It  seem that there  are  various flavours of the Hyers-Ulam stability studied  in the literature. However, the concept studied in~\cite{BBT2, BRST, BLR} precisely corresponds to our notion of shadowing. In a series of remarks, we will now show how our results extend and unify those established in the papers we mentioned.

\begin{remark}
In~\cite{BBT2}, the authors prove that if $X=\mathbb C^m$ and if $A$ is an hyperbolic  operator on $X$(i.e  its spectrum doesn't intesect the unit circle), then~\eqref{ld2} with $A_n=A$, $n\ge 0$ has the $l^\infty$-Lipschitz shadowing property.
This result is a particular case of our Corollary~\ref{957} since the constant sequence $(A_n)_{n\ge 0}$ admits an exponential dichotomy. 
\end{remark}

\begin{remark}
In~\cite{BRST}, the authors study the system~\eqref{ld2} when $(A_n)_{n\ge 0}$ is a $q$-periodic sequence of linear operators on $X=\mathbb C^m$. They prove that~\eqref{ld2} has the $l^\infty$-shadowing property if $\cA(q,0)=A_{q-1}\cdots A_0$ is hyperbolic. 
Since the hyperbolicity of $\cA(q,0)$ implies that $(A_n)_{n\ge 0}$ admits an exponential dichotomy, this result is also a particular case of our Corollary~\ref{957}.
\end{remark}
\begin{remark}
Consider two sequences $(a_n)_{n\ge 0}$ and $(b_n)_{n\ge 0}$ in $\mathbb C$ and the associated  linear recurrence
\begin{equation}\label{lr}
x_{n+2}=a_nx_{n+1}+b_n x_n, \quad n\ge 0.
\end{equation}
Set 
\[
A_n=\begin{pmatrix}
0 & 1 \\
b_n & a_n
\end{pmatrix} \quad \text{for  $n\ge 0$,}
\]
and consider the associated linear system in $\mathbb C^2$ given by
\begin{equation}\label{ld3}
y_{n+1}=A_n y_n, \quad n\ge 0.
\end{equation}
Observe that if $(x_n)_{n\ge 0}\subset \mathbb C$ is a solution of~\eqref{lr} then  $(y_n)_{n\ge 0}$ given by $y_n=\begin{pmatrix} x_n \\ x_{n+1} \end{pmatrix}$ is a solution of~\eqref{ld3}. 
Conversely, if $(y_n)_{n\ge 0}$, $y_n=\begin{pmatrix} y_n^1\\ y_n^2\end{pmatrix}$ is a solution of~\eqref{ld3}, then $(x_n)_{n\ge 0}$ given by $x_n=y_n^1$ is a solution of~\eqref{lr} and 
$y_n^2=y_{n+1}^1$ for each $n\ge 0$.

Assume that the sequence $(A_m)_{m\ge 0}$ admits an exponential dichotomy and let us consider the norm $\lVert \cdot \rVert$ on $\mathbb C^2$ given by
\[
\lVert (z_1, z_2)\rVert:=\max \{ \lvert z_1\rvert, \lvert z_2\rvert \}. 
\]
Take $\varepsilon >0$ and let us consider $\delta >0$ that corresponds to $l^\infty$-Lipschitz shadowing of~\eqref{ld3}. We now take a sequence $(w_n)_{n\ge 0} \subset \mathbb C$ such that 
\[
\sup_{n\ge 0}\lvert w_{n+2}-a_nw_{n+1}-b_n w_n\rvert \le \delta. 
\]
Set $z_n=\begin{pmatrix} w_n \\ w_{n+1} \end{pmatrix}$, $n\ge 0$. Then, $(z_n)_{n\ge 0}$ is an $(\delta, l^\infty)$-pseudotrajectory. Hence, Corollary~\ref{957} implies that  there exists $(y_n)_{n\ge 0}$ solution of~\eqref{ld3} such that
$\sup_{n\ge 0}\lVert y_n-z_n\rVert \le \epsilon$. Hence, $(x_n)_{n\ge 0}$  given by $x_n=y_n^1$ is a solution of~\eqref{lr} and 
$\sup_{n\ge 0}\lvert x_n-w_n\rvert \le \sup_{n\ge 0}\lVert x_n-z_n\rVert\le \epsilon$. We conclude that~\eqref{lr} also has an $l^\infty$-Lipschitz shadowing property.
Consequently, since we have not assumed that the sequences $(a_n)_{n\ge 0}, (b_n)_{n\ge 0}$ are periodic, this gives a partial generalization of \cite[Theorem 2.3]{BLR}. 
\end{remark}

We hope that the results and the ideas developed in the present paper could be of use to establish additional results related to Hyers-Ulam stability.

\subsection{Grobman-Hartman's theorem} As an other application of our results we obtain a new proof of the nonautonomous version of the classical Grobman-Hartman theorem~\cite{Hart60}. More precisely, we revisit~\cite[Section 4.1]{BD19} to apply our new results in order to show that our ideas can be used to obtain a less restrictive version of~\cite[Theorem 4.1]{BD19}.

Let $(A_m)_{m\in \Z}$ be a sequence of bounded linear operators on $X$ as in Subsection~\ref{sec: setup}. Furthermore, suppose that each $A_m$ is invertible and that $\sup_{m\in \Z}\lVert A_m^{-1}\rVert <\infty$. Associated to these parameters by Theorem \ref{theo: dichotomy uniqueness} (applied to $B=l^\infty$ and $f_n\equiv 0$), consider $\varepsilon>0$ sufficiently small and $\delta=L\varepsilon>0$. Let $(g_n)_{n\in \Z}$ be a sequence of maps $g_n\colon X\to X$ satisfying \eqref{fg} with $c$ sufficiently small and such that
\[\lVert g_n\rVert_{\sup} \le \delta \quad  \text{for each $n\in \Z$.} \] 
We consider a  difference equation
\begin{equation}\label{eq: rec GH}
y_{n+1}=G_n(y_n) \quad n\in \Z,
\end{equation}
where $G_n:=A_n+g_n$. By decreasing $c$ (if necessary), we have that $G_n$ is a homeomorphism for each $n\in \Z$ (see~\cite{BD19}). Then, we have the following result. 

\begin{theorem}\label{theo: GH}
There exists a unique sequence $h_m \colon X \to X$, $m\in \Z$, of homeomorphisms
such that  for each $m\in \Z$,
\begin{equation}\label{eq:gh1}
h_{m+1}\circ G_m=A_m \circ h_m
\end{equation}
and
\begin{equation}\label{eq:gh2}
\lVert h_m-\Id\rVert_{\sup}=\sup_{x\in X} \lVert h_m(x)-x\rVert \le \epsilon. 
\end{equation}
\end{theorem}
The family of homeomorphism $h_m \colon X \to X$, $m\in \Z$, satisfying \eqref{eq:gh1} and \eqref{eq:gh2} is constructed ``explicitly" using the $l^\infty$-Lipschitz shadowing property. In fact, fix  $m\in \Z$. Given $y\in X$, let us consider the sequence $\mathbf y=(y_n)_{n\in \Z}$ given by
$y_n=\mathcal{G}(n, m)y$ for $n\in \Z$ where
\[
\mathcal{G}(m,n)=
\begin{cases}
G_{m-1}\circ \ldots \circ G_n & \text{if $m>n$,}\\
\Id & \text{if $m=n$,}\\
G_m^{-1}\circ \ldots \circ G_{n-1}^{-1} & \text{if $m<n$.}
\end{cases}
\]
Then, $\mathbf y$ is a solution of~\eqref{eq: rec GH}. Moreover,
\begin{displaymath}
\sup_{n\in \Z}\lVert y_{n+1}-A_n y_n\rVert=\sup_{n\in \Z} \lVert g_n(y_n)\rVert \le
\delta.
\end{displaymath}
In particular, $\mathbf y=(y_n)_{n\in \Z}$ is a $(\delta, l^\infty)$-pseudotrajectory for \eqref{ld}. Hence, it follows from Theorem \ref{theo: dichotomy uniqueness} applied to the case when $B=l^\infty$ and $f_n\equiv 0$ that there exists a unique sequence $\mathbf x=(x_n)_{n\in \Z}$ such that $x_{n+1}=A_nx_n$ for $n\in \Z$ and $\sup_{n\in \Z}\lVert x_n-y_n \rVert \le
\varepsilon$. Set
\[
 h_m(y)=h_m(y_m):=x_m. 
\]
It is easy to verify that~\eqref{eq:gh1} holds. Moreover, 
\[
\lVert h_m(y)-y\rVert=\lVert x_m-y_m\rVert \le \epsilon, 
\]
proving ~\eqref{eq:gh2}. It remains to show that each $h_m$ is a homeomorphism. The proof of this fact is similar to the proof of Theorem 4.1 of \cite{BD19} and thus is left as an exercise. We also refer to Remark 4.3 of \cite{BD19} for references to related results. The difference from the aforementioned result and our Theorem \ref{theo: GH} is that this last result works under less restrictive assumptions since the nonlinear perturbations allowed in Theorem \ref{theo: dichotomy uniqueness} are much more general than the ones in \cite{BD19}. Indeed, the assumptions in Theorem~\ref{theo: GH} coincide with those in~\cite{Palmer} where the first nonautonomous version of the Grobman-Hartman theorem was obtained (although Palmer studied dynamics with continuous time).


\medskip{\bf Acknowledgements.}
We would like to thank the anonymous referee for his/hers constructive comments  that helped us to improve the quality of the presentation. 
 L.B. was partially supported by a CNPq-Brazil PQ fellowship under Grant No. 306484/2018-8. D.D. was supported in part by Croatian
Science Foundation under the project IP-2019-04-1239 and by the University of
Rijeka under the projects uniri-prirod-18-9 and uniri-prprirod-19-16.

\end{document}